\documentclass[12pt,a4paper]{amsart}

\usepackage{amsmath, amssymb, amsthm}
\usepackage{geometry}
\usepackage[hidelinks]{hyperref}
\usepackage{mathrsfs}
\usepackage{color}

\usepackage{amsfonts, mathrsfs, amscd}

\usepackage{amsbsy}
\usepackage{amscd}
\usepackage[mathscr]{eucal}

\allowdisplaybreaks
\usepackage{verbatim}

\theoremstyle{definition}
\newtheorem{theorem}{Theorem}[section]
\newtheorem*{conjecture}{LG-LG mirror symmetry conjecture}

\newtheorem{lemma}[theorem]{Lemma}
\newtheorem{definition}[theorem]{Definition}
\newtheorem{remark}[theorem]{Remark}

\newtheorem{fjraxiom}{FJR}

\newcommand{\SL}{{\mathrm{SL}}}

\newcommand{\Hom}{\mathrm{Hom}}

\newcommand{\Jac}{\mathrm{Jac}}
\newcommand{\Fix}{\mathrm{Fix}}

\newcommand{\id}{\mathrm{id}}

\newcommand{\bs}{{\bf s}}
\newcommand{\bz}{{\bf z}}

\def\A{{\mathcal A}}

\def\C{{\mathcal C}}

\def\E{{\mathcal E}}
\def\F{{\mathcal F}}

\def\H{{\mathcal H}}

\def\M{{\mathcal M}}

\def\O{{\mathcal O}}

\def\T{{\mathcal T}}

\def\Hom{{\mathrm{Hom}}}
\def\MF{{\mathrm{MF}}}

\def\p{\partial }
\def\ns{{\nabla}\hspace{-1.4mm}\raisebox{0.3mm}{\text{\footnotesize{\bf /}}}}

\newcommand{\im}{\mathsf{i}}
\newcommand{\sH}{\H}
\DeclareMathOperator{\cFix}{Fix}
\newcommand{\jw}{\mathfrak{j}}
\newcommand{\set}[1]{\left\{#1\right\}}  % a set
\newcommand{\br}[1]{\left\langle#1\right\rangle}  % angle brackets

\newcommand{\ccHH}{{\mathsf{HH}}}

\newcommand{\RR}{{\mathbb{R}}}
\newcommand{\CC}{\mathbb{C}}

\newcommand{\HC}{\mathrm{HC}}

\newcommand{\GM}{{\mathrm{GM}}}

\renewcommand{\S}{{\mathcal S}}
\def\calS{{\mathcal S}}

\title{Landau--Ginzburg mirror symmetry for ADE singularities}
\author{Alexey Basalaev}
\date{\today}

\begin{document}
\maketitle

\begin{abstract}
We prove Landau--Ginzburg mirror symmetry conjecture for ADE singularities with the nontrivial symmetry groups.
In particular, we consider $(f,G)$ with $f \in \CC[z_1,z_2,z_3]$ defining ADE singularity and $G$ --- nontrivial group of diagonal symmetries of $f$.
We show that orbifold Saito theory of $(f,G)$ is isomorphic to the Fan-Jarvis-Ruan-Witten theory of its Berglund--H\"ubsch--Henningson dual pair $(\widetilde f, \widetilde G)$.
\end{abstract}

%   \setcounter{tocdepth}{1}
%   \tableofcontents

\section{Introduction}

Mirror symmetry for Landau--Ginzburg models relates two different theories associated to a quasihomogeneous singularity. On the A-side, the Fan--Jarvis--Ruan--Witten (abbreviated by FJRW) theory associates to a polynomial together with a suitable group of its diagonal symmetries a cohomological field theory. On the B-side, K.~Saito's theory of primitive forms equips the universal unfolding of an isolated hypersurface singularity with the structure of a Frobenius manifold. Together with the Givental formalism, this gives the corresponding Saito--Givental theory. In a most general context not only the A-side, but the B-side as well should be considered with the certain choice of a symmetry group.
A fundamental problem in Landau--Ginzburg mirror symmetry is to understand when these two theories are related by a mirror isomorphism.

\subsection{FJRW and Saito theories}
Saito theory is associated to a polynomial $f$, defining an isolated singularity. This is an important branch of singularity theory developed in the last quarter of XX (cf. \cite{SK1, SK2, SK3, SM1, SM2}).

FJRW theory is associated to a quasihomogeneous polynomial $f$ defining an isolated singularity, equipped with an additional data - a group of symmetries $G$ of the polynomial $f$. Namely, FJRW theory is associated to a pair $(f,G)$. In the case of $f$ being ADE polynomials, FJRW theory gave solution to the big Witten's conjecture (cf. \cite{FJR13, FJR07}).

Despite the fact that both Saito and FJRW theories have a polynomial defining isolated singularity in their input, two theories are totally different. Saito theory is analytical, while FJRW is an example of modern enumerative theory. The unifying notion that allows one to compare two is {\it Dubrovin--Frobenius manifold} (\cite{Dubrovin96}). This was the observation of B.~Dubrovin that Saito theory of ADE singularities allows one to construct a Dubrovin--Frobenius manifold. FJRW theory defines a Dubrovin--Frobenius manifold by taking the genus $0$ part of its cohomological field theory. In what follows we assume FJRW theory to be isomorphic to Saito theory if their Dubrovin--Frobenius manifolds are isomorphic.

\subsection{Orbifold Saito theory}
This is transparent from the beginning that the A-side and B-side of mirror symmetry have non-symmetric input. Namely, FJRW theory is associated to a pair $(f,G)$ while the Saito theory is associated to a polynomial $f$ itself. Modern developmend of mirror symmetry suggested that the B--side should be considered with the symmetry group as well. There were many attempts to construct an ``orbifold'' version of Saito theory, associated to a pair $(f,G)$ (cf. \cite{BT22, BTW23, BR26}). In this note we use the construction of Junwu Tu \cite{T21a,T21b} in order to define the B--side Dubrovin--Frobenius manifold of $(f,G)$. We choose to call it \textit{orbifold Saito theory}. This follows immediately from the constuction that our results hold true for the definition of \cite{BR26} as well.

\subsection{Landau--Ginzburg mirror symmetry conjecture}
Let
\[
f(z_1,\ldots,z_N) = \sum_{i=1}^N c_i \prod_{j=1}^N z_j^{e_{ij}}, \qquad c_i\in\mathbb C^*,
\]
be a polynomial with $N$ monomials in $N$ variables. It's called an \emph{invertible polynomial} if the matrix $E_f = \{e_{ij} \}_{1\leq i,j\leq N}$ is invertible and $f$ is nondegenerate, i.e. has an isolated critical point at the origin.

The maximal group of diagonal symmetries of $f$ is
\[
G_f := \left\{ (\lambda_1,\ldots,\lambda_N)\in(\mathbb C^*)^N \ | \ f(\lambda_1z_1,\ldots,\lambda_Nx_N)=f(z_1,\ldots,z_N) \right\}.
\]
It's always non-empty for invertible $f$. For any $G \subseteq G_f$ the pair $(f,G)$ is called a Landau--Ginzburg model or just LG model.

Associate to $f$ a new polynomial $\widetilde f$ by transposing the exponent matrix
\[
\widetilde f(z_1,\ldots,z_N)
:=
\sum_{i=1}^N
\prod_{j=1}^N z_j^{e_{ji}}.
\]
Namely, $E_{\widetilde f} = E_f^T$. This construction is due to Berglung and H\"ubsch \cite{BH93}.
For any $G \subseteq G_f$ Berglund and Henningson defined (cf. \cite{BH95}) the dual group $\widetilde G \subseteq G_{\widetilde f}$ by
$
\widetilde G = \Hom(G_f/G,\CC^*).
$
Thus the BHH duality, named after the inventors Berglund, H\"ubsch and Henningson reads
\[
(f,G) \longleftrightarrow (\widetilde f,\widetilde G).
\]

\begin{conjecture}
Let the polynomial $f$ define an invertible singularity and let
$G\subseteq G_f \cap \SL(N,\CC)$.

Then the orbifold Saito theory of $(f, G)$ is isomorphic to the FJRW theory of $(\widetilde f, \widetilde G)$.
\end{conjecture}

For $f$ defining ADE singularity and $G = \{ \id \}$, this conjecture was proved by \cite{FJR13,HLSW, HPSV, MS16, LLS17, K10}.
The purpose of the present paper is to extend this result to nontrivial symmetry groups. Our main result is the following.

\begin{theorem}
LG-LG mirror symmetry conjecture holds for ADE singularities.
\end{theorem}

Namely, we extend the previously known ADE cases with trivial symmetry group to all nontrivial symmetry groups.

Writing ADE singularity in three variables all non--trivial admissible symmetry groups were considered in \cite{BTW23}. For every such pair $(f,G)$ the authors of loc.cit. constructed the polynomials $\overline f$ such that the pair $(f,G)$ is orbifold equivalent to $(\overline f ,\{ \id \})$ (see Section~5 of \cite{BTW23})). We list these polynomials in Table~\ref{tableADE}.

We prove our mirror theorem in two steps.

\begin{theorem}\label{theorem: explicit}
 For every pair $(f,G)$ from Table~\ref{tableADE} holds:

 \begin{enumerate}
 \item[(i)] the orbifold Saito theory of $(f,G)$ is isomorphic to classical Saito theory of $\overline f$.

 \item[(ii)] the FJRW theory of $(\widetilde f, \widetilde G)$ is isomorphic to the classical Saito theory of $\overline f$.
 \end{enumerate}
\end{theorem}

We prove part~(i) in Section~\ref{section: proof of part i} and part (ii) in Section~\ref{section: proof of part ii}.

In \cite{K10} M.~Krawitz proved that LG-LG mirror symmetry conjecture holds as the vector space isomorphism. In \cite{FJJS} A.~Francis, T.~Jarvis, D.~Jhonson and R.~Suggs that LG-LG mirror symmetry conjecture holds as the isomorphism of Frobenius algebras. To our knowledge the theorem above is the first LG-LG mirror symmetry isomorphism of Dubrovin--Frobenius manifolds where the B--model is given with the nontrivial symmetry group. Next to \cite{BT22} this is the second example of such mirror symmetry isomorphism.

\begin{remark}
Cohomological field theory of the orbifold Saito theory is not defined. However this can be done by employing A.~Givental's reconstruction \cite{G01}. This is an easy consequence of the theorem above and the result of C.~Teleman \cite{Tc12} that LG mirror symmetry conjecture holds as the isomorphism of cohomological field theories as well.
\end{remark}

\begin{table}[h]
{\small
\begin{center}
\begin{tabular}{l||l|c||c| c}
&$f(z_1,z_2,z_3)$ & $G$ & $\overline{f}(y_1,y_2,y_3)$ & Type
\\
\hline
1.&$z_1^{k+1}+z_2^2+z_3^2$,\quad $k \ge  1$ & $\left<\frac{1}{2}(0,1,1)\right>$ & $y_1^{k+1}+y_2+y_2y_3^2$ & $A_k \times A_k$
\\
2.&$z_1^{2k}+z_2^2+z_3^2$,\quad $k \ge  1$ & $\left<\frac{1}{2}(1,0,1)\right>$ & $y_1^2+y_2^k+y_2y_3^2$ & $D_{k+1}$
\\
3.&$z_1^{2k}+z_2^2+z_3^2$,\quad $k \ge  1$ & $\left<\frac{1}{2}(0,1,1),\frac{1}{2}(1,0,1)\right>$ & $y_1^ky_2^k+y_1y_3+y_2y_3$ & $A_{2k-1}$
\\% $y_1^k+y_1y_2+y_2y_3^2$\\
4.&$z_1^{3}+z_2^3+z_3^2$ & $\left<\frac{1}{3}(1,2,0)\right>$ & $y_1^2+y_3y_2^2+y_2y_3^2$ & $D_4$
\\
5.&$z_1^{4}+z_2^3+z_3^2$ & $\left<\frac{1}{2}(1,0,1)\right>$ & $y_1^3+y_2^2+y_2y_3^2$ & $E_6$
\\
6.&$z_1^2+z_2^2+z_2z_3^{2k}$,\quad $k \ge  1$ & $\left<\frac{1}{2}(1,0,1)\right>$ &  $y_1^2+y_1y_2^k+y_2y_3^2$ & $D_{2k+1}$
\\
7.&$z_1^2+z_2^2+z_2z_3^{2k+1}$,\quad $k \ge  1$ & $\left<\frac{1}{2}(0,1,1)\right>$ & $y_1^2+y_3y_2^2+y_2y_3^{k+1}$ & $D_{2k+2}$
\\
8.&$z_1^2+z_2^{k-1}+z_2z_3^2$,\quad $k \ge  4$ & $\left<\frac{1}{2}(1,0,1)\right>$ & $y_1^{k-1}+y_1y_2+y_2y_3^2$ & $A_{2k-3}$
\\
\hline
\end{tabular}
\end{center}
}
\smallskip
\caption{$(f,G) \sim_{orb} (\overline{f},\{\id\})$}\label{tableADE}
\end{table}

\subsection{Acknowledgements}
The research leading to these results has received funding from the Basic Research Program at the National Research University Higher School of Economics.

\section{Dubrovin--Frobenius manifold}
A Dubrovin--Frobenius manifold is the set of data $(M,\circ,\eta,e)$, where $M$ is a complex manifold, $\circ$ is a fiberwise $\O_M$--bilinear commutative and associative product on its holomorphic tangent sheaf $\T_M$, $\eta$ is a non--degenerate $\O_M$--bilinear form on $\T_M$, such that the Frobenius algebra property $\eta(X \circ Y, Z) = \eta(X,Y\circ Z)$ holds for any $X,Y,Z \in \T_M$. It is required that the Levi-Civita connection $\ns$ of $\eta$ is flat, and the unit vector field $e$ of the product is $\ns$--flat. Let $t_\alpha$ be the corresponding flat coordinates and $\p / \p t_\alpha$ the dual vector fields. Denote by $C := \sum_\alpha \left( \frac{\p}{\p t_\alpha} \circ \right) dt_\alpha$ the $\T_M$--endomorphism-valued $1$--form. It is required that $\ns C = 0$. Moreover both $\eta$ and $\circ$ are required to be quasihomogeneous with respect to some Euler vector field $E$ (see \cite{Dubrovin96}).

The structure of every Dubrovin--Frobenius manifold can be encoded by just one function $\F = \F(t_1,\dots,t_\mu)$ called its potential.
Denote by $\eta^{ij}$ the components of $\eta^{-1}$.
The associativity of the product $\circ$ then implies that $\F$ is subject to a big system of PDEs called the \textit{WDVV equation}.
\begin{equation}
\sum_{a,b}
\frac{\partial^3\mathcal F}
{\partial t^i\partial t^j\partial t^a}
\eta^{ab}
\frac{\partial^3\mathcal F}
{\partial t^b\partial t^k\partial t^\ell}
=
\sum_{a,b}
\frac{\partial^3\mathcal F}
{\partial t^i\partial t^k\partial t^a}
\eta^{ab}
\frac{\partial^3\mathcal F}
{\partial t^b\partial t^j\partial t^\ell},
\label{eq:wdvv}
\end{equation}
that should hold for any fixed $(i,j,k,l)$.

\subsection{Dubrovin connection}
To any Dubrovin--Frobenius manifold one can associate the connection $\nabla$ on $M \times \CC^\ast_z$ given by
\begin{align}\label{eq: Dubrovin connection-1}
    \nabla_X := \ns_X + \frac{1}{z} C_X, \quad \nabla_{\frac{d}{dz}} := \frac{d}{dz} + \frac{1}{z} \left( B_0 + \frac{B_\infty}{z} \right)
\end{align}
for $C_X(Y) := X \circ Y$, $B_0(Y) := E \circ Y$ and the diagonal grading operator $B_\infty$.

This new connection is flat. It provides an important piece of data of a Dubrovin--Frobenius manifold.
The converse way to construct a Dubrovin--Frobenius manifold from the meromorphic connection is given by the VSHS of S. Barannikov (see \cite{B00,B01,Sc07}).

\subsection{Reconstruction lemma}
Denote $V := T_0M = \CC\langle e_1,\dots,e_\mu \rangle$, the tangent space at $t_1=\dots=t_\mu = 0$.
The $k$--point correlator of a Dubrovin--Frobenius manifold is the complex number
\[
    \langle e_{i_1},\dots,e_{i_k} \rangle^\F_{0,k} := \frac{\p^k \F}{\p t_{i_1} \dots \p t_{i_k} } \mid_{t_1 = \dots = t_\mu = 0}.
\]
The three--point correlators endow $V$ with the $\CC$--algebra structure
\[
    e_i \circ e_j = \sum_{k,m=1}^\mu \langle e_i, e_j, e_m \rangle_{0,3} \eta^{mk} e_k.
\]
This is important to stress that $\circ$ does not depend on $t$.

The Euler vector field gives the grading $\deg$ on $V$. This grading is compatible with the product.

\begin{definition}
The vector $\gamma \in V$ is called \emph{primitive} if there is no $\gamma_1,\gamma_2 \in V$, such that  $\gamma = \gamma_1 \circ \gamma_2$ and $0 < \deg(\gamma_1) \le \deg(\gamma_2) < \deg(\gamma)$. We call a correlator $\langle \dots \rangle_{0,n}^\F$ \emph{basic} if it involves at most two non-primitive insertions.
\end{definition}

The following lemma will be used later on.

\begin{lemma}[Lemma~6.2.8 in \cite{FJR13}]\label{lemma: reconstruction in genus zero}
Fix a Dubrovin--Frobenius manifold with the central charge $\hat c$.
If $\deg(\alpha) \le \hat c$ for all vectors $\alpha \in V$ then all correlators of this Dubrovin--Frobenius manifold are uniquely determined by $\eta$ and the $n$-point correlators with $n \le N_{min}$.
\[
    N_{min} := \left\lfloor 2 + \frac{1 + \hat c}{1 - P} \right\rfloor, \quad P:= \max_{\substack{v \in V \\ v \text{ is primitive}}} \deg(v).
\]
\end{lemma}

The proof of this lemma is based on the analysis of WDVV equation.

Concerning ADE Saito theory Dubrovin--Frobenius manifolds Fan, Jarvis and Ruan proved the following important result using the reconstruction lemma above.

\begin{lemma}[Lemma~6.2.9 in \cite{FJR13}]
 All the correlators $\langle \dots \rangle_{0,n}^\F$ for the ADE singularities, in either the A-model or the B-model, are uniquely determined by the pairing, the three-point correlators, and the four-point correlators.
\end{lemma}

\section{LG models}\label{section: LG models}
Let $f\in \CC[z_1,\dots,z_N]$ be a quasihomogeneous polynomial of degree $d$ with integer weights $w_1, \dots, w_N$ such that $\gcd(w_1, \dots, w_N)=1$. Namely,
\[
    f(\lambda^{2 \pi \im w_1}z_1,\dots,\lambda^{2 \pi \im w_N}z_N) = \lambda^d f(z_1,\dots,z_N), \quad \forall \lambda \in \CC^\ast.
\]
For each $1\leq k\leq N$, let $q_k=\tfrac{w_k}{d}$. The central charge of $f$ is defined to be
\[
\hat c:=\sum_{k=1}^N(1-2q_k).
\]
A polynomial is \emph{nondegenerate} if
\begin{itemize}
\item[(i)] the weights $q_k$ are uniquely determined by $f$, and
\item[(ii)] the hypersurface defined by $f$ is non-singular in projective space.
\end{itemize}

$f$ is said to define a quasihomogeneous singularity if it's quasihomogeneous and nondegenerate. Such singularities were investigated in details in \cite{SK1}. The special class of such singularities is given by the so--called invetible singularities. These are given by the polynomials $f$ that can be written as a sum of exactly $N$ monomials none of which is $z_iz_j$.

The class of invertible singularities is distinguished in mirror symmetry because one can propose a mirror pair to it, however these are obviously not the only singularities that are under investigation in mirror symmetry (see \cite{HC10,BI24,R24,BT22}).

The following polynomials define ADE singularities.
\begin{align*}
    & f = z_1^{k+1}, \ A_k; \quad f = z_1^{k-1} + z_1z_2^2, \ D_k;
    \\
    & f = z_1^3 + z_2^4, \ E_6; \quad f = z_1^3 + z_1z_2^3, \ E_7, \quad
    f = z_1^3 + z_2^5, \ E_8.
\end{align*}

The maximal group of diagonal symmetries of $f$ is
\[
    G_f := \left\{ (\lambda_1,\ldots,\lambda_N) \in (\CC^*)^N \ \middle|\
f(\lambda_1z_1,\ldots,\lambda_Nz_N)=f(z_1,\ldots,z_N)
\right\}.
\]
In what follows we will denote its elements $(e^{2\pi \im \alpha_1},\dots,e^{2\pi \im \alpha_N})$ just by $(\alpha_1,\dots,\alpha_N)$. Every quasihomogeneous $f$ has a distinguished element $\jw \in G_f$ given by
\[
    \jw = (q_1,\dots,q_N).
\]
The group generated by $\jw$ will be denoted by $J_f$ or just $J$.

For any $G \subseteq G_f$ the pair $(f,G)$ is called a Landau--Ginzburg model or just LG model. In the LG-LG mirror symmetry conjecture both the A-side and B-side are fixed by some LG models. However, the B--side LG model $(f,G)$ is additionally supposed to satisfy $G \subseteq \SL(N,\CC)$ and the A--side LG model $(\widetilde f, \widetilde G)$ is supposed to satisfy $J_{\widetilde f} \subseteq \widetilde G$. This is a huge distinction showing that the A-side can not be considered with the trivial group $G = \{ \id \}$.

\begin{remark}
In this note we only consider diagonal symmetry groups. However there are many publications working with the nonabelian symmetry groups in mirror symmetry as well (cf. \cite{EG22, PWW20, BI21, B25}).
\end{remark}

For the ADE Saito theory the algebra $(V,\circ)$ of Lemma~\ref{lemma: reconstruction in genus zero} coincides with
$\Jac(f) := \mathbb C[z_1,\ldots,z_N] / (\partial_{z_1}f,\ldots,\partial_{z_N})$, local algebra of $f$.
Denote $\mu :=\dim_{\mathbb C}\Jac(f)$, Milnor number of the singularity.

\subsection{Orbifold equivalence}\label{section: orb equivalence}
In this note we are going to use very special orbifold equivalent pairs. For the general notion we refer the reader to \cite{CRCR16,RCN16}. We follow the ideas of \cite{BTW23,BT20}.

Let $f = f(z_1,z_2,z_3)$ and $N=3$ and $G\subseteq G_f \cap \SL(3,\CC)$. Consider the new polynomial $\overline f \in \CC[y_1,y_2,y_3]$ defined as follows.

Denote by $\widehat{\CC^3/G}$ a crepant resolution of $\CC^3/G$. The polynomial $f$ gives the map $\widehat f: \widehat{\CC^3/G} \to \CC$. It was observed in \cite{BTW23} that for all $(f,G)$ defining the B--side LG model just one of the crepant resolution charts will contain all the critical points of $\widehat f$. Denote by $\overline f$ the corresponding polynomial representing $\widehat f$ in this chart.

Then we have the following equivalence of the categories of matrix factorizations.
\begin{equation*}
\MF(\overline f)=\MF_{\{\id\}}(\overline f)\cong \MF_G(f).
\end{equation*}

This follows due to the local property of the category of singularities
of $\{\widehat{f}=0\}\subset \widehat{\CC^3/G}$ which is equivalent to $\MF_G(f)$ \cite[Proposition~1.14]{Orl1}.

All pairs $(f,G)$ with $f$ defining ADE singularity in $3$ variables, nontrivial $G \subseteq G_f \cap \SL(3,\CC)$ and the corresponding $\overline f$ are given in Table~\ref{tableADE}. Theorem~64 of \cite{BTW23} also establishes explicit isomorphisms
\[
    \Jac(\overline f) \cong \Jac(f,G) = \ccHH^\ast \left( \MF_G(f) \right).
\]

\subsection{Singularities equivalence}
In classical singularity theory one considers the polynomials $f$ up to the left-right equivalence. Namely, two polynomials $f_1,f_2$ are said to be left-right equivalent (and define the same singularity) if they can be connected by the invertible transformations of their domains. In particular, $D_4$ singularity can be defined either by $f = z_1^3 + z_1z_2^2$ or by $f = z_1^3 + z_2^3$.

Similarly, the polynomials $f(z_1,\dots,z_N)$ and $f^{st} := f(z_1,\dots,z_N) + z_{N+1}^2$ are often assumed to be equivalent. The singularity theory operation connecting two polynomials is called ``stabilization'' rathen than equivalence. In particular, the local algebras of two such polynomials are canonically isomorphic.

This is yet another point that will distinguish A-side and B-side constructions associated to the LG models. In particular, all B-side constructions of this note will be invariant under the stabilization $f \to f^{st}$, while the same property does not in general hold for the A--side.

\section{Variations of semi-infinite Hodge structures: the B side}

Let $M$ be a formal or analytic parameter space and let $\O_M[[u]]$ denote the sheaf of formal power series in a variable $u$.

A VSHS over $M$ consists of the data $(\E, \nabla, K)$: a locally free $\O_M[[u]]$-module $\E$, a meromorphic flat connection $\nabla: \mathcal E \longrightarrow \mathcal E\otimes_{\O_M} \Omega^1_M((u))$
and a pairing $K: \E \otimes_{\O_M} \E \longrightarrow \O_M[[u]]$. Assume also

\begin{enumerate}
    \item $\E$ is of finite rank, equipped with a $\mathbb{Z}/2\mathbb{Z}$-grading.

    \item $\nabla$ is required to have restricted pole structures at $u = 0$.
        \[
            \nabla_{\partial_u}: \E \longrightarrow u^{-2}\E, \quad \nabla_X : \E \longrightarrow u^{-1}\E \quad \forall X \in \T_M.
        \]

    \item $K$ is non-degenerate and $\O_M$-bilinear, compatible with the grading.
    It satisfies the $u$-sesquilinear property $K( u \alpha, \beta) = K(\alpha, -u \beta)$
    for any $\alpha, \beta \in \E$, and is compatible with the connection $\nabla$.
\end{enumerate}

\subsection{Primitive form of a VSHS}

The data of VSHS is more general than the Dubrovin--Frobenius manifold. In order to relate two we need to use primitive forms.

Given a VSHS $(\E,\nabla, K)$, a primitive form is a distinguished section $\zeta\in \E$ such that
\begin{equation*}
\mathcal T_M \longrightarrow \E/u\E, \qquad X \longmapsto u\nabla_X\zeta \pmod{u\E}.
\end{equation*}
is an isomorphism.
This isomorphism is used to lift the connection $\nabla$ to $M \otimes \CC_z^\ast$. If the element $\zeta$ satisfies the certain list of additional properties, this lifted connection plays the role of Dubrovin connection and endows $M$ with the structure of a Dubrovin--Frobenius manifold (see \cite{Sc07, B00, B01}).

\subsection{Saito theory VSHS}

\subsubsection{Unfolding of a singularity}

Let $ f:\CC^N \longrightarrow \CC $ have an isolated critical point at the origin, and let
$ \Jac(f)$ be its local algebra, $\mu:=\dim_{\mathbb C}\Jac(f)$.

An \textit{unfolding} of $f$ is the function $F : Z \to \CC$, where $Z = \CC^N \times \calS$, for some open neighbourhood of the origin $\calS \subset \CC^\mu$ with coordinates $s_\bullet$, defined by
\[
    F(\bz,\bs) = f(\bz) + \sum_{k = 1}^\mu \phi_k(\bz) s_k.
\]
Here the polynomials $\phi_\bullet \in \CC[\bz]$ are taken to be such that their classes generate $\Jac(f)$.

Denote by $p: Z \to \calS$ the projection on the second factor.
Consider the critical sheaf
\[
    \O_\C := \O_{Z} / \left( \p_{z_1}F, \dots, \p_{z_N} F \right).
\]
Then $p_\ast \O_\C$ is an $\O_\S$--module of rank $\mu$ with the essential product structure. The product of the Saito theory Dubrovin--Frobenius manifold will be inherited exactly from $p_\ast \O_\C$.

Let $(\Omega^\bullet_{Z/\calS},d_{Z/\calS})$ stand for the de Rham complex relative to $p$. Consider the direct image sheaf
\[
    \RR^k p_\ast \left( \Omega^\bullet_{Z/\calS} [u],u d_{Z/\calS} + dF \wedge \right).
\]
It vanishes for $k \neq N$. And when $k=N$, it is isomorphic to the sheaf
\[
    \H_F^{(0)} := \Omega^N_{Z/\calS} [u] / (u d_{Z/\calS} + dF \wedge ) \Omega^{N-1}_{Z/\calS} [u].
\]
This sheaf is called {\it Brieskorn lattice of $F$}. It is a locally free $\O_\calS[u]$--module of rank $\mu$ (cf. Proposition 3.5 of \cite{IMRS}).

\subsubsection{Gauss--Manin connection and higher residue pairing}\label{section: Gauss-Manin connection}
Consider the completion $\H_F := \H_F^{(0)} \otimes \CC(u)$.
It is endowed with the Gauss--Manin connection $\nabla^\GM$ defined as follows. Let $v \in \T_\calS$, $\phi \in \O_\calS[u]$, and $d^N\bz = dz_1\wedge \dots \wedge dz_N$. Denote by $v(\phi)$ the directional derivative along $v$ and by $[\phi d^N\bz]$ the class of $\phi d^N\bz$ in $\H_F$.
\begin{align*}
    & \nabla_v^\GM [\phi d^N\bz] := \left[ (v (\phi) + u^{-1} \phi \cdot v(F)) d^N\bz \right],
    \\
    & \nabla_{\frac{\p}{\p u}}^\GM [\phi d^N\bz] := \left[ \left(\frac{\p \phi}{\p u} - u^{-2} \phi \cdot F - \frac{N}{2} u^{-1} \phi \right) d^N\bz \right].
\end{align*}
The Gauss-Manin connection satisfies the Leibniz rule and is flat.

K. Saito introduced the pairing (cf. \cite{SK2})
\[
   K_F: \H_F^{(0)}\otimes_{\O_\S} \H_F^{(0)}\to \O_\S[u]
\]
called the \textit{higher residue pairing}, uniquely defined by the following properties. Firs of all $K_F$ is flat with respect to the Gauss--Manin connection.
Let
\[
 K_F(\omega_1, \omega_2) = \sum_{p \ge 0} u^p K_F^{(p)}(\omega_1,\omega_2)
\]
then we have $K_F^{(p)}(\omega_1, \omega_2) = (-1)^p K_F^{(p)}(\omega_2,\omega_1)$, $K_F(u\omega_1, \omega_2) = -K_F(\omega_1, u \omega_2) = uK_F(\omega_1,\omega_2)$. The leading term
$K_F^{(0)}$ defines the pairing
$$
    \H_F^{(0)}/u  \H_F^{(0)}\otimes  \H_F^{(0)}/u  \H_F^{(0)} \to \CC, \quad \omega_1\otimes \omega_2\mapsto K_F^{(0)}(\omega_1,\omega_2)
$$
which coincides with the residue pairing $\eta$.

This is a theorem of K.~Saito that these properties fix the higher residue pairing completely.
The higher residue pairing $K_F$ extends to the completion $\H_F$.

\subsubsection{The VSHS}

K. Saito developed in \cite{SK3} the theory of primitive forms $\zeta \in \H_F^{(0)}$ as the special elements useful to construct the period mapping. Existence of a primitive form was proven by M. Saito in \cite{SM1,SM2}. It was later observed that having fixed the primitive form, the space $\S$ can be endowed with the Dubrovin-Frobenius manifold structure (see \cite{ST}). We can formulate their theorem as follows.

\begin{theorem}[\cite{ST}]\label{theorem: classical saito}
    For any $f$ defining an isolated singularity the data $(\H_F, \nabla^{\GM},K_F)$ is a VSHS on the unfolding space $\S$. After the choice of a primitive form of Saito this VSHS produces the structure of a Dubrovin--Frobenius manifold on $\S$.
\end{theorem}

\subsection{Categorical Saito theory}\label{section: categorical saito theory}
Let $\C_{f,G}$ stand for the $A_\infty$--category $\MF_G(f)$.

Fix a basis $\psi_1,\dots, \psi_\delta$ of $\ccHH^\ast(\MF_G(f))$ and the formal parameters $t_1,\dots,t_\delta$ corresponding to these basis elements.
Let $\C_{f,G,t}$ be the family of $A_\infty$--categories given by the formal universal deformation $\beta_{f,G}$. This is a solution to the Maurer--Cartan equation on $\C_{f,G}$, such that $\beta_{f,G} = \beta_1 + \beta_2 + \dots$ with $\beta_1 = \sum_{i=1}^\delta t_i \psi_i$.

\begin{remark}
    This is important to stress that Saito theory is associated to the unfolding $F$ that may be viewed as a deformation of the polynomial $f$ by the set of polynomials $\phi_1,\dots,\phi_\mu$. The context of categorical Saito theory is more general. In particular, the elements $\psi_1,\dots, \psi_\delta$ above are not just polynomials, but the cohomology elements already, that might even have different grading.
\end{remark}

Associate to the pair $(f,G)$ the data
\begin{equation*}
\left( \HC^-_\ast(\C_{f,G,t}),
\nabla^{f,G},
\langle-,-\rangle_{\mathrm{Muk}}
\right).
\end{equation*}

Here $\HC^-_\ast(\C_{f,G,t})$ is negative cyclic cohomology of $\C_{f,G,t}$. This negative cyclic cohomology is endowed with the Getzler-Gauss-Manin connection $\nabla^{f,G}$ and Mukai pairing $\langle-,-\rangle_{\mathrm{Muk}}$. We do not give the definitions becasue these are lengthy and refer the reader to \cite{Sn20}.

\begin{theorem}[\cite{T21b, T21a, AT22}]\label{theorem: categorical saito equivariant}
    The data $\left(\HC^-_\ast(\C_{f,G,t}), \nabla^{f,G}, \langle-,-\rangle_{\mathrm{Muk}}\right)$ is a VSHS for any LG model $(f,G)$.
    If $f$ is invertible, this VSHS can be endowed with a primitive element $\zeta$, that can be used to construct the Dubrovin--Frobenius manifold.
\end{theorem}

We will denote this Dubrovin--Frobenius manifold by $M_{f,G}^\zeta$.
The theorem above also holds for $G=\{\id\}$.

\begin{theorem}[\cite{Sd14, Sd16, Sd20, ST}]\label{theorem: categorical saito classical}
    For any $f$ defining the quasihomogeneous singularity, the data $\left(\HC^-_\ast(\C_{f,\{\id\},t}), \nabla^{f,\{\id\}}, \langle-,-\rangle_{\mathrm{Muk}}\right)$ is a VSHS.
    This VSHS is isomorphic to the VSHS of the classical Saito theory of $f$.
\end{theorem}

\subsection{Proof of Theorem~\ref{theorem: explicit}, part (i)}\label{section: proof of part i}

In order to prove part (i) of Theorem~\ref{theorem: explicit} it's enough  to establish the isomorphism of two VSHS. The first one of $\overline f$ given in Theorem~\ref{theorem: classical saito} and the second of $(f,G)$ given in Theorem~\ref{theorem: categorical saito equivariant}.

By using Theorem~\ref{theorem: categorical saito classical} we may consider both VSHS given in the categorical construction of Section~\ref{section: categorical saito theory}. This allows us to use the methods of homological algebra.

By the construction the pair $(f,G)$ is orbifold equivalent to $\overline f$ and there is equivalence of categories (see Section~\ref{section: orb equivalence})
\begin{equation*}
\Phi: \MF_G(f) \longrightarrow \MF(\bar f).
\end{equation*}

Morita invariance of negative cyclic homology gives an isomorphism
\begin{equation*}
\Phi_{\mathrm{HC}^-}: \HC^-_\ast(\MF_G(f)) \xrightarrow{\sim} \HC^-_\ast(\MF(\overline f)).
\end{equation*}
We may assume that this isomorphism agrees with the fixed isomorphism $\Jac(f,G) \cong \Jac(\overline f)$ of \cite{BTW23}.
This extends $\Phi_{\mathrm{HC}^-}$ to the isomorphism of the formal parameter spaces of two VSHS.

The isomorphism $\Phi$ is compatible with the Getzler--Gauss--Manin connection (cf. Theorem~A and Theorem~B.2 of \cite{Sn20}) and the pairing
\begin{equation*}
\Phi_{\mathrm{HC}^-} \circ \nabla^{f,G} = \nabla^{\bar f,\{\id\}} \circ \Phi_{\mathrm{HC}^-},
\quad
\left\langle \Phi_{\mathrm{HC}^-}(a), \Phi_{\mathrm{HC}^-}(b) \right\rangle_{\mathrm{Muk}} = \langle a,b\rangle_{\mathrm{Muk}}.
\end{equation*}

Thus the categorical VSHS is a Morita invariant --- the VHSH defined by $(f,G)$ and $(
\overline f, \{ \id \})$ are isomorphic by $\Phi_{\mathrm{HC}^-}$.

Therefore the choice of a primitive form for both can be identified by $\Phi_{\mathrm{HC}^-}$ giving the isomorphic Dubrovin--Frobenius manifolds.
This completes proof of part (i).

\section{FJRW theory: the A side}\label{section:definitionOfFJRW}

The Landau--Ginzburg A model is provided by FJRW theory. Its input is a pair $(W,G)$ of a quasihomogeneous polynomial and a group $G \subseteq G_{W}$ satisfying $J_W \subseteq G$.

\subsection{State Space}
Originally  FJRW theory defines a state space and a moduli space of $W$-curves, from which one obtains certains numbers---called correlators---as integrals over the moduli space (see \cite{FJR07, FJR13}).

We do not give all details in this note introducing only the CohFT of FJRW theory.

For $h\in G$, let $\cFix(h)$ denote the fixed locus of $\CC^N$ with respect to $h$, let $N_h$ denote the dimension of $\cFix(h)$ and let $W_h$ denote $W|_{\cFix(h)}$.

For any $h \in G$ let $\Omega_{\cFix(h)}^{m}$ stand for the sheaf of holomorphic $m$--forms of $\cFix(h)$.
Define
\begin{equation*}
\sH_{h}' := \Omega^{N_h}_{\cFix(h)} / dW_h \wedge \Omega^{N_h-1}_{\cFix(h)}.
\end{equation*}
This is a general fact from singularity theory that $\sH_{h}'$ has dimension equal to the Milnor number of $W_h$. It is also endowed with the pairing --- residue pairing.
Let the group $G$ act on the sheaf of holomorphic $m$--forms coordinate-wise.
Define
\begin{equation*}
\sH_{h} := \left( \sH_{h}' \right)^G
\end{equation*}
that is, $G$-invariant elements of $\sH_{h}'$. The state space is the direct sum of the ``sectors'' $\sH_h$, i.e.
\[
\sH_{W,G} := \bigoplus_{h\in G}\sH_h.
\]

Let $G^{nar}=\set{h\in G\mid N_h=0}$. These summands $\sH_h$ for $h\in G^{nar}$ are the so-called \emph{narrow} sectors. The summands with the positive--dimensional fixed locus are called {\it broad}.

$\sH_{W,G}$ is $\mathbb{Q}$-graded by the $W$-degree. To define this grading, first note that each element $h\in G$ can be uniquely expressed as a tuple of rational numbers
$h=(\Theta_1^h,\dots,\varTheta_N^h)$
with $0\leq \varTheta_k^h < 1$.
We first define the \emph{degree-shifting number}
\[
\iota(h):=\sum_{k=1}^N(\varTheta_k^h-q_k),
\]
where $q_1,\dots,q_N$ are the rational numbers, giving the quasihomogeneity weights of $W$ (see Section~\ref{section: LG models}).
For $\alpha_h \in \H_{h}$, the (real) $W$-degree of $\alpha_h$ is defined by
\begin{equation*}
\deg_W(\alpha_h):=N_h+2\iota(h).
\end{equation*}

The sector indexed by $\jw$, is one--dimensional, and has $W$--degree 0.
This sector is unique with this property.

Because $\cFix(h)= \cFix(h^{-1})$ there is a nondegenerate pairing
\[
\br{-,-}:\sH_h\times\sH_{h^{-1}}\to \CC,
\]
the residue pairing of $W_h$, which induces a symmetric nondegenerate pairing
\[
\br{- ,- }:\sH_{W,G}\times \sH_{W,G}\to \CC.
\]

\subsection{Axioms of FJRW theory}\label{ss:FJRWaxioms}
FJRW theory associates to the pair $(W,G)$ the CohFT $c_{g,n}^{W,G}$ on the moduli space of $n$--pointed genus $g$ stable curves $\overline \M_{g,n}$ (see Theorem~4.2.2 of \cite{FJR13}).
We do not give the construction of this CohFT but list some properties that it satisfies. We follows the common convention to call these properties axioms (see Section~4 of \cite{FJR13}).

Assumed $\alpha_1,\dots, \alpha_n \in \H_{W,G}$ with $\alpha_k \in \H_{h_k}$ for some $h_k \in G$.

\begin{fjraxiom}[Degree]\label{a:fjr1}
The class $c_{g,n}^{W,G}(\alpha_1,\dots,\alpha_n)$ has degree
\[
2\left((\hat c -3)(1-g)+ n -\sum_{i=1}^n \iota(h_i)\right).
\]
In particular, if this number is not an integer, then the class vanishes.
\end{fjraxiom}

\begin{fjraxiom}[Line bundle degree]\label{a:line bundle}
If the numbers
\[
l_k := q_k(2g-2+n)-\sum_{i=1}^n \varTheta_k^{h_i}, \quad k=1,\dots,N
\]
are not integer the class $c_{g,n}^{W,G}(\alpha_1,\dots,\alpha_n)$ vanishes.
\end{fjraxiom}

\begin{fjraxiom}[$G_W$-invariance]
The class $c_{g,n}(\alpha_1,\dots,\alpha_n)$ is invariant under the coordinate-wise action of $G_W$ on $\H_{W,G}$.
\end{fjraxiom}

\begin{fjraxiom}[Concavity]\label{a:fjrlast}
Suppose that $\Fix(h_k) = 0$ and $l_k < 0$ for all $1 \le k \le N$, then
\begin{align*}
    c_{g,n}^{W,G}(\alpha_1,\dots,\alpha_n) &= \sum_{l=1}^N \Big[ \left( \frac{q_l^2}{2} - \frac{q_l}{2} + \frac{1}{12} \right) \kappa_1 - \sum_{i=1}^N \left(\frac{1}{12} - \frac{1}{2} \Theta_l^{k_i} (1 - \Theta_l^{h_i}) \right) \psi_i
    \\
    & + \frac{1}{2} \sum_\Gamma \left(\frac{1}{12} - \frac{1}{2} \Theta_l^{h_+} (1 - \Theta_l^{h_+}) \right) [ \overline \M(\Gamma)] \Big]
\end{align*}
where the last summation is taked over the modular graphs $\Gamma$ of genus $g$ with $n$ leaves decorated by $\alpha_1,\dots,\alpha_n$.
\end{fjraxiom}

\begin{fjraxiom}[Sum of singularities]\label{a:fjrSum}
Let $W = W_1 + W_2$ and $G = G_1 \times G_2$ for the LG models $(W_1,G_1)$ and $(W_2,G_2)$ such that $W_1, W_2$ depend on the non-intersecting sets of variables. Then $\alpha_k = \alpha_k' \alpha_k''$ and
\[
 c_{g,n}^{W,G}(\alpha_1,\dots,\alpha_n) = c_{g,n}^{W_1,G_1}(\alpha_1',\dots,\alpha_n') \cdot c_{g,n}^{W_2,G_2}(\alpha_1'',\dots,\alpha_n'')
\]
assumed as the product of cohomology classes on the moduli space of curves.
\end{fjraxiom}

In the list above we did not include such propoerties of $c_{g,n}^{W,G}$ as Index zero axiom, Symmetric Group invariance and Deformation invariance. This is just because we are not going to use them in what follows. Also the way we introduce concavity axiom is special. This is the result of additional computation made with the help of Chiodo formula. See Theorem~6.3.3 of \cite{FJR13}.

Let $\alpha_1,\dots,\alpha_\nu$ be a basis of $\H_{W,G}$ and $t_1,\dots,t_\nu$ stand for the formal variables associated to these vectors.
The Dubrovin--Frobenius manifold of the FJRW theory is fixed by the potential
\[
    \F = \F(t_1,\dots,t_\nu) := \sum_{n \ge 3} \frac{1}{n!} \int_{\overline{\M}_{0,n}} c_{0,n}^{W,G} \left(T, \dots, T \right),
\]
with $T = \sum_i t_i \alpha_i$. Namely, this is just the genus zero potential of a CohFT.

\begin{theorem}[Theorem 6.1.3 of \cite{FJR13}, Theorem 1.1 of \cite{FFJMR} and Theorem 3 of \cite{FSZ}]
\label{theorem: FJR-ADE}
    For any $k \ge 1$ let $f$ be one of
    $f = z_1^{k+1}$, $f = z_1^{2k+2} + z_1z_2^2$, $f = z_1^3 + z_2^4$, $f = z_1^3 + z_1z_2^3$, $f = z_1^3 + z_2^5$
    defining $A_k$, $D_{2k+2}$, $E_6$, $E_7$ or $E_8$ singularities respectively. Then FJRW theory of $(f,J)$ is isomorphic to Saito theory of $f$.
\end{theorem}

The theorem above is a LG-LG mirror symmetry theorem for $A_k$, $E_6$ and $E_8$ type singularities but is not a LG-LG mirror symmetry theorem for $D_{2k+2}$ and $E_7$ type singularities because in the latter cases $f \neq \widetilde f$ and $\widetilde J$ is not a trivial group.

In the theorem above it is very important that the polynomials defining the singularities are fixed. In particular, we will see later on that FJRW theories of $(z_1^k,J)$ and $(z_1^k + z_2^2,J)$ are not isomorphic even though the two polynomials are related by the stabilization (see Section~\ref{section: LG models}).

\section{FJRW theory computations}\label{section: proof of part ii}

In this section we prove part (ii) of Theorem~\ref{theorem: explicit}. Comparing to the proof of part (i) of this theorem we need to do the computations case by case. We also work with the Dubrovin--Frobenius manifold by themselves rather than the VSHS.

We use the reconstruction Lemma~\ref{lemma: reconstruction in genus zero}. Denote by $\A_{\widetilde f, \widetilde G}$ the corresponding $\CC$--algebra given by FJRW correlators.
In order to apply reconstruction lemma we need to find first the primitive vectors and the number $N_{min}$. This data is fully determined by the algebra $\A_{\widetilde f, \widetilde G}$ and the quasihomogeneity. In all our cases the algebra isomorphism $\A_{\widetilde f, \widetilde G} \cong \Jac(f, G)$ holds by \cite{FJJS}. By \cite{BTW23,BT20} we have $\Jac(f, G) \cong \Jac(\overline f)$. We use this to figure out primitive elements and the number $N_{min}$.

The proof is given case by case following the numeration of Table~\ref{tableDuals}.

FJRW theory is not easy to compute. Most of the results computing FJRW theory are for the maximal symmetry groups $G_f$, exceptions being given by \cite{F12,B22,G17,FHS23}.
In this note we need indeed to compute it with the nonmaximal groups. Like in the loc.cit. main tool of ours to compute broad insertions is WDVV equation.

\begin{table}[h]
{\small
\begin{center}
\begin{tabular}{l||c|c||c| c}
&$f$ & $G$ & $\widetilde{f}$ & $\widetilde{G}$
\\
\hline
1.&$z_1^{k+1}+z_2^2+z_3^2$,\quad $k \ge  1$ & $\left<\frac{1}{2}(0,1,1)\right>$ & $z_1^{k+1}+z_2^2+z_3^2$ & $\langle (\frac{1}{k+1},0,0) \rangle \times \langle (0,\frac{1}{2},\frac{1}{2}) \rangle$
\\
2.&$z_1^{2k}+z_2^2+z_3^2$,\quad $k \ge  1$ & $\left<\frac{1}{2}(1,0,1)\right>$ & $z_1^{2k}+z_2^2+z_3^2$ & $\langle (\frac{1}{2k},0,\frac{1}{2}) \rangle \times \langle (0,\frac{1}{2},0) \rangle$
\\
3.&$z_1^{2k}+z_2^2+z_3^2$,\quad $k \ge  1$ & $\left<\frac{1}{2}(0,1,1),\frac{1}{2}(1,0,1)\right>$ & $z_1^{2k}+z_2^2+z_3^2$ & $J_{\widetilde{f}}$
\\
4.&$z_1^{3}+z_2^3+z_3^2$ & $\left<\frac{1}{3}(1,2,0)\right>$ & $z_1^{3}+z_2^3+z_3^2$ & $J_{\widetilde{f}}$
\\
5.&$z_1^{4}+z_2^3+z_3^2$ & $\left<\frac{1}{2}(1,0,1)\right>$ & $z_1^{4}+z_2^3+z_3^2$ & $J_{\widetilde{f}}$
\\
6.&$z_1^2+z_2^2+z_2z_3^{2k}$,\quad $k \ge  1$ & $\left<\frac{1}{2}(1,0,1)\right>$ &  $z_1^2+z_2^2z_3+z_3^{2k}$ & $J_{\widetilde{f}}$
\\
7.&$z_1^2+z_2^2+z_2z_3^{2k+1}$,\quad $k \ge  1$ & $\left<\frac{1}{2}(0,1,1)\right>$ & $z_1^2+z_2^2z_3+z_3^{2k+1}$ & $\langle (\frac{1}{2},0,0) \rangle \times \langle (0,\frac{4k+1}{4k+2},\frac{1}{2k+1}) \rangle$
\\
8.&$z_1^2+z_2^{k-1}+z_2z_3^2$,\quad $k \ge  4$ & $\left<\frac{1}{2}(1,0,1)\right>$ & $z_1^2+z_2^{k-1}z_3+z_3^2$ & $J_{\widetilde{f}}$
\\
\hline
\end{tabular}
\end{center}
}
\smallskip
\caption{The pairs $(\widetilde f, \widetilde G)$}\label{tableDuals}
\end{table}

\subsection{Case 1}
We have to compute FJRW theory of the polynomial $f = \widetilde f = z_1^{k+1} + z_2^2 + z_3^2$ with the symmetry group $\widetilde G = \langle (\frac{1}{k+1},0,0) \rangle \times \langle (0,\frac{1}{2},\frac{1}{2}) \rangle$.

We have $f = f_1 + f_2$ for $f_1 = z_1^{k+1}$ and $f_2 = z_2^2 + z_3^2$. By the sum of singularities axiom, FJRW theory of $(\widetilde f,\widetilde G)$ is isomorphic to the product of the FJRW theories of $(f_1,J_{f_1})$ and $(f_2, J_{f_2})$. The first theory is isomorphic to Saito theory of $A_k$ by Theorem~\ref{theorem: FJR-ADE}.
FJRW theory of $(f_2, J_{f_2})$ is two--dimemsional, isomorphic to the product of Saito theories $A_1 \times A_1$.

This completes the proof.

\subsection{Case 2}
We have to compute FJRW theory of the polynomial $f = \widetilde f = z_1^{2k} + z_2^2 + z_3^2$ with the symmetry group $\widetilde G = \langle (\frac{1}{2k},0,\frac{1}{2}) \rangle \times \langle (0,\frac{1}{2},0) \rangle$.

By the sum of singularities axiom and triviality of the $2$--spin FJRW theory, it's enough to compute FJRW theory of the polynomial $f(z_1, z_2) = z_1^{2k} + z_2^2$ with the symmetry group $J$.
It's known that $\A_{\widetilde f,\widetilde G} \cong \Jac(D_{k+1})$ (see \cite{FJR13, FJJS}) but we check that again.

\subsubsection{State Space}
We have
\begin{equation*}
    \hat{c} = \frac{k-1}{k}, 
    \quad \H_{f, J} = \bigoplus_{c=1}^{2k} \H_{f, \jw^c}
\end{equation*}

For $c = 2a$ with $a = 1,\dots,k-1$ we have $\H_{f, \jw^c} = 0$.
When $c = 2a+1$ for $a = 0, \dots, k-1$, the fixed locus of $\jw^c$ is the origin and $\H_{f, \jw^c} \cong \CC$. We denote the corresponding basis vector by $\phi_a$. The W--degree is:
\begin{equation*}
    \deg(\phi_a) = \frac{2a+1}{2k} - \frac{1}{2k} + \frac{1}{2} - \frac{1}{2} = \frac{a}{k}.
\end{equation*}

The sector $c = 2k$ is broad with the fixed locus $\CC^2$. Imposing $J$-invariance we have $\H_{f, \jw^{2k}} \cong \CC$ generated by $[z_1^{k-1}dz_1\wedge dz_2]$ that we denote by $\tau$. Its degree is:
\begin{equation*}
    \deg(\tau) = \frac{k-1}{2k} + 1 - \frac{1}{2k} - \frac{1}{2} = \frac{k-1}{2k}.
\end{equation*}

Fix the basis for the state space $\{\phi_0, \dots, \phi_{k-1}, \tau\}$.

\subsubsection{Frobenius Algebra}
By the degree axiom a three-point primary correlator $\langle \gamma_1, \gamma_2, \gamma_3 \rangle_{0,3}$ is non-zero only if $\sum_{i=1}^3 \deg(\gamma_i) = \hat{c} = \frac{k-1}{k}$.
The only non-zero three-point correlators are:
\begin{align*}
    & \langle \phi_a, \phi_b, \phi_c \rangle_{0,3} = 1 \quad \text{if} \quad a+b+c = k-1, \\
    & \langle \phi_0, \tau, \tau \rangle_{0,3} = 1.
\end{align*}
Note that these formulae also fix the pairing because $\phi_0$ is the unit by the construction.

This gives $\phi_a \circ \phi_b = \phi_{a+b}$, $\phi_a \circ \tau = 0$ for $a > 0$, and $\tau \circ \tau = \phi_{k-1}$.

\subsubsection{Dubrovin--Frobenius manifold}
Primitive are $\phi_1$, $\tau$ and $N_{min} = 5$. Repeating word by word the proofs of Lemma~6.2.9 and Theorem 6.2.10 part (3) from \cite{FJR13} we get that genus zero FJRW theory we care about is uniquely determined by the pairing, three-point and the following four-point correlators
\[
    \langle \phi_1, \phi_1, \phi_1^{k-1}, \phi_1^{k-2} \rangle_{0,4} \quad \text{and} \quad \langle \tau,  \tau, \phi_1, \phi_1^{2} \rangle_{0,4} \ \text{if} \ k=3.
\]
These correlators satisfy selection rule axiom. 
The first correlator is concave evaluating to
\[
    \langle \phi_1, \phi_1, \phi_1^{k-1}, \phi_1^{k-2} \rangle_{0,4}  = \frac{1}{2k}.
\]

For the case $k=3$ we have one more correlator to compute. It involves broad insertions. From the WDVV equation we have for any vector $Y$:
\begin{align*}
    & \sum_X \left( \langle \tau,\tau,Y,X \rangle_{0,4} \langle \overline{X}, \phi_1, \phi_1 \rangle_{0,3}
    + \langle \tau,\tau,X \rangle_{0,3} \langle \overline{X}, Y, \phi_1, \phi_1 \rangle_{0,4} \right)
    \\
    &\quad =
    \sum_X \left( \langle \tau,\phi_1,Y,X \rangle_{0,4} \langle \overline{X}, \tau, \phi_1 \rangle_{0,3}
    + \langle \tau,\phi_1,X \rangle_{0,3} \langle \overline{X}, Y, \tau, \phi_1 \rangle_{0,4} \right),
\end{align*}
where $\overline X$ stands for the dual to $X$ with respect to the pairing.

Take $Y := \phi_{k-2}$.
Both summands of the second line vanish due to the selection rule axiom. The summation of the first line simplifies to
\begin{align*}
    & \langle \tau,\tau,\phi_{k-2},\phi_2 \rangle_{0,4} \langle \phi_{k-3}, \phi_1, \phi_1 \rangle_{0,3} + \langle \tau,\tau,\phi_0 \rangle_{0,3} \cdot \frac{1}{2k} = 0.
\end{align*}
Taking $k=3$ we arrive at
\[
    \langle \tau,\tau,\phi_1,\phi_2 \rangle_{0,4}
    = - \frac{1}{6}.
\]

This follows now from part (3) of Theorem~6.2.10 in \cite{FJR13}
that FJRW theory in question is isomorphic to Saito theory of $D_{k+1}$.
%%%%%%%%%%%%%%%%%%%%%%%%%%%%%%%%%%%%%%%%%%%%%%%%%%%%%%%%%%%%%%%%%%%%%%%%%%%%%%%%%%%%%%
%%%%%%%%%%%%%%%%%%%%%%%%%%%%%%%%%%%%%%%%%%%%%%%%%%%%%%%%%%%%%%%%%%%%%%%%%%%%%%%%%%%%%%
%%%%%%%%%%%%%%%%%%%%%%%%%%%%%%%%%%%%%%%%%%%%%%%%%%%%%%%%%%%%%%%%%%%%%%%%%%%%%%%%%%%%%%

\subsection{Case 3}
We compute FJRW of $f = \widetilde f = z_1^{2k} + z_2^2 + z_3^2$ with the symmetry group $J$.

\subsubsection{State Space and Frobenius Algebra}
The central charge is:
\begin{equation*}
    \hat{c} = 1 - \frac{2}{2k} = \frac{k-1}{k}.
\end{equation*}
The state space $\H_{f, J}$ has a complex dimension of $2k-1$. We fix the basis
\begin{equation*}
    \{ \phi_0, \phi_1, \dots, \phi_{2k-2} \}.
\end{equation*}
The W--degrees are:
\begin{equation*}
    \deg(\phi_a) = \frac{a}{2k}, \quad \text{for } a = 0, \dots, 2k-2.
\end{equation*}

The product $\circ$ and three-point primary correlators map identically to the standard $A_{2k-1}$ local algebra governed by the dimension constraint $\sum_{i=1}^3 \deg(\phi_i) = \hat{c}$.
The only non-zero three-point correlators are
\begin{align*}
    & \langle \phi_a, \phi_b, \phi_c \rangle_{0,3} = 1 \text{ where } a+b+c = 2k-2,
\end{align*}
what gives the product $\phi_a \circ \phi_b = \phi_{a+b}$.

\subsubsection{Dubrovin--Frobenius manifold}
Primitive is only $\phi_1$ and $N_{min} = 4$. Genus zero FJRW theory in question is uniquely determined by the pairing, three-point and the following four-point correlator
\[
    \langle \phi_1, \phi_1, \phi_1^{2k-2}, \phi_1^{2k-2} \rangle_{0,4}.
\]
This correlator satisfies selection rule axiom. It is concave evaluating to $-\frac{1}{2k}$.

This follows now from part (1) of Theorem~6.2.10 in \cite{FJR13}
that FJRW theory in question is isomorphic to Saito theory of $A_{2k-1}$.

%%%%%%%%%%%%%%%%%%%%%%%%%%%%%%%%%%%%%%%%%%%%%%%%%%%%%%%%%%%%%%%%%%%%%%%%%%%%%%%%%%%%%%
%%%%%%%%%%%%%%%%%%%%%%%%%%%%%%%%%%%%%%%%%%%%%%%%%%%%%%%%%%%%%%%%%%%%%%%%%%%%%%%%%%%%%%
%%%%%%%%%%%%%%%%%%%%%%%%%%%%%%%%%%%%%%%%%%%%%%%%%%%%%%%%%%%%%%%%%%%%%%%%%%%%%%%%%%%%%%

\subsection{Case 4}
We compute FJRW of $f = \widetilde f = z_1^{3} + z_2^3 + z_3^2$ with the symmetry group $J$.

Note that $\widetilde f = f_1 + f_2$ for $f_1 = z_1^3 + z_2^3$, $f_2 = z_3^2$ and $J_{\widetilde f} = J_{f_1} \times J_{f_2}$. By the sum of singularities axiom FJRW theory of $(\widetilde f, J)$ is the product of FJRW theories of $(f_1,J_{f_1})$ and $(f_2,J_{f_2})$. The latter theory is trivial. The polynomial $f_1$ is right--left equivalent by a quasihomogeneous transformation to the polynomial $x^3 + xy^2$. This polynomial defines $D_4$ singularity and the FJRW theory of it with the group $J$ was investigated in \cite{FFJMR}. This follows from Theorem~4.5 and Lemma~5.2 in loc.cit. that FJRW theory in question is isomorphic to Saito theory of $D_4$ singularity.

%%%%%%%%%%%%%%%%%%%%%%%%%%%%%%%%%%%%%%%%%%%%%%%%%%%%%%%%%%%%%%%%%%%%%%%%%%%%%%%%%%%%%%
%%%%%%%%%%%%%%%%%%%%%%%%%%%%%%%%%%%%%%%%%%%%%%%%%%%%%%%%%%%%%%%%%%%%%%%%%%%%%%%%%%%%%%
%%%%%%%%%%%%%%%%%%%%%%%%%%%%%%%%%%%%%%%%%%%%%%%%%%%%%%%%%%%%%%%%%%%%%%%%%%%%%%%%%%%%%%

\subsection{Case 5}
We compute FJRW of $f = \widetilde f = z_1^4 + z_2^3 + z_3^2$ with the symmetry group $J$.

The state space $\H_{\widetilde f,J}$ is 6-dimensional, $\widehat c = \frac{1}{3}$.

The direct summands $\H_{\jw^1}$, $\H_{\jw^5}$, $\H_{\jw^7}$, $\H_{\jw^{11}}$ are narrow and hence $1$--dimensional. Denote their generators by $\alpha_1, \alpha_3,\alpha_4,\alpha_6$ respectively.
Additionally let $\alpha_2 = [z_1 dz_1 \wedge dz_2]$ be the generator of $\H_{\jw^4}$ and $\alpha_5 = [z_1 dz_1 \wedge dz_2]$ be the generator of $\H_{\jw^8}$.

\subsubsection{Frobenius algebra}
The non-zero 3-point correlators are:
\[
     \langle \alpha_1, \alpha_1, \alpha_6 \rangle_{0,3} = 1, \ \langle \alpha_1, \alpha_2, \alpha_5 \rangle_{0,3} = 1, \ \langle \alpha_1, \alpha_3, \alpha_4 \rangle_{0,3} = 1, \ \langle \alpha_2, \alpha_2, \alpha_3 \rangle_{0,3} = 1.
\]
This gives the product structure isomorphic to $\mathbb{C}[x,y]/(x^3, y^2)$ by mapping $\alpha_2 \mapsto x$ and $\alpha_3 \mapsto y$.
Primitive are $\alpha_2,\alpha_3$ and $N_{min} = 4$. By the section rule and degree axiom the following basic correlators are enough to define the genus zero theory
\[
    c_1 := \langle \alpha_2, \alpha_2, \alpha_2^2, \alpha_2^2\alpha_3\rangle_{0,4}, \quad
    c_2 := \langle \alpha_3, \alpha_3, \alpha_2\alpha_3, \alpha_2\alpha_3 \rangle_{0,4}.
\]

\begin{lemma}
    We have $c_1^2 = 1$ and $c_2 = \frac{1}{3}$.
\end{lemma}
\begin{proof}
In order to compute these correlators note that $\widetilde f = f_1 + f_2$ for $f_1 = z_1^4 + z_3^2$, $f_2 = z_2^3$ and $J_{\widetilde f} = J_{f_1} \times J_{f_2}$. By the sum of singularities axiom FJRW theory of $(\widetilde f, J)$ can be computed via the product of FJRW theories of $(f_1,J_{f_1})$ and $(f_2,J_{f_2})$. The latter theory is fully computed by Theorem~\ref{theorem: FJR-ADE} and we compute the first one.

{\bf Step 1.}
The state space $\H_{f_1,J}$ is three--dimensional generated by $A = 1$ in the $\jw$--sector, $B = [z_1 dz_1\wedge dz_3]$ in the $\jw^0$--sector and $C = 1$ in the $\jw^3$--sector. Three--potin correlators are all known by \cite{FJJS}. By the degree and selection rule axioms the higher puncture correlators that can be non--zero are the folllowing
\[
 \langle B, B, C, C \rangle_{0,4}, \quad \langle C, C, C, C, C \rangle_{0,5}.
\]
The five-point correlator above can be computed by the concavity axiom to be equal to $2$. Associating to $A, B, C$ the variables $t_A,t_B,t_C$ we get the following Dubrovin-Frobenius potential
\[
\F(t_A,t_B,t_C) = \frac{1}{2} t_A^2 t_C + \frac{1}{2} t_A t_B^2 + \frac{\lambda}{4} t_B^2 t_C^2 + \frac{1}{60}t_C^5,
\]
where $\lambda$ is the value of the four--point correlator above. WDVV equation on $\F$ above is equivalent to $\lambda^2 = 1$.

{\bf Step 2.}
Let $e_1,e_2$ be the standard basis vectors of the $3$--spin FJRW theory.
By the sum of singularities axiom we have the following equality in the cohomology ring $H^\ast(\overline{\M}_{0,4})$
\begin{align*}
    c_{0,4}^{\widetilde f}(\alpha_2, \alpha_2, \alpha_2^2, \alpha_2^2\alpha_3) =
    c_{0,4}^{f_1}(B, B, C, C) \cdot c_{0,4}^{f_2}(e_1, e_1, e_1, e_2).
\end{align*}
Decomposing the degree $1$ class on the left hand side as the product of degree $1$ and $0$ classes on the right hand side we get.
\begin{align*}
    \langle \alpha_2, \alpha_2, \alpha_2^2, \alpha_2^2\alpha_3 \rangle^{\widetilde f}_{0,4}
    =
    \langle B, B, C, C \rangle_{0,4}^{f_1} \cdot \langle e_1, e_1, e_2 \rangle_{0,3}^{f_2} \langle e_1,e_1, e_2\rangle_{0,3}^{f_2} = \lambda.
\end{align*}
Similarly
\begin{align*}
     c_{0,4}^{\widetilde f}(\alpha_3, \alpha_3, \alpha_2\alpha_3, \alpha_2\alpha_3)
    =  c_{0,4}^{f_1}(A, A, B, B) \cdot c_{0,4}^{f_2}(e_2, e_2, e_2, e_2).
\end{align*}
Decomposing the degree $1$ class on the left hand side as the product of degree $0$ and $1$ classes on the right hand side we get.
\begin{align*}
     \langle \alpha_3, \alpha_3, \alpha_2\alpha_3, \alpha_2\alpha_3 \rangle^{\widetilde f}_{0,4}
     =
     \langle A, A, C \rangle_{0,3}^{f_1}
     \langle A, B, B \rangle_{0,3}^{f_1} \cdot
     \langle e_2, e_2, e_2, e_2 \rangle_{0,4}^{f_2} = \frac{1}{3}.
\end{align*}
\end{proof}

Applying the rescaling $\alpha_2 \to \lambda 2^{1/2} \alpha_2$, $\alpha_4 \to 2 \alpha_4$ and $\alpha_5 \to \lambda 2^{1/2} \alpha_5$ we get the correlators that coincide with those for $E_6$ FJRW theory with the group $J$ as computed in \cite{FJR13} after the division by two. In the other words two Dubrovin--Frobenius manifold structures coincide, but the potentials of them are connected by the multiplication by $2$.

This follows now from part (4) of Theorem~6.2.10 in \cite{FJR13}
that FJRW theory in question is isomorphic to Saito theory of $E_6$.

%%%%%%%%%%%%%%%%%%%%%%%%%%%%%%%%%%%%%%%%%%%%%%%%%%%%%%%%%%%%%%%%%%%%%%%%%%%%%%%%%%%%%%
%%%%%%%%%%%%%%%%%%%%%%%%%%%%%%%%%%%%%%%%%%%%%%%%%%%%%%%%%%%%%%%%%%%%%%%%%%%%%%%%%%%%%%
%%%%%%%%%%%%%%%%%%%%%%%%%%%%%%%%%%%%%%%%%%%%%%%%%%%%%%%%%%%%%%%%%%%%%%%%%%%%%%%%%%%%%%

\subsection{Case 6}
We compute FJRW theory of $\widetilde f = z_1^2 + z_2^2 + z_2 z_3^{2k}$ with the symmetry group~$J$.

Applying to $\widetilde f$ the coordinate transformation $w_2 = z_2 + \frac{1}{2} z_3^{2k}$, we get the polynomial:
\begin{equation*}
    W(z_1, w_2, z_3) = z_1^2 + w_2^2 - \frac{1}{4} z_3^{4k}.
\end{equation*}
This coordinate transformation is quasihomogeneous and therefore FJRW theories of $(\widetilde f, J_{\widetilde f})$ and of $(W, J_W)$ are isomorphic.

The FJRW theory $(W, J_W)$ was investigated in Case~(2) above.

%%%%%%%%%%%%%%%%%%%%%%%%%%%%%%%%%%%%%%%%%%%%%%%%%%%%%%%%%%%%%%%%%%%%%%%%%%%%%%%%%%%%%%
%%%%%%%%%%%%%%%%%%%%%%%%%%%%%%%%%%%%%%%%%%%%%%%%%%%%%%%%%%%%%%%%%%%%%%%%%%%%%%%%%%%%%%
%%%%%%%%%%%%%%%%%%%%%%%%%%%%%%%%%%%%%%%%%%%%%%%%%%%%%%%%%%%%%%%%%%%%%%%%%%%%%%%%%%%%%%

\subsection{Case 7}
Compute FJRW theory of $\widetilde f = z_1^2+z_2^2z_3+z_3^{2k+1}$ with the symmetry group $\widetilde G: = G_1 \times G_2$, where $G_1 := \langle (\frac{1}{2},0,0) \rangle$ and $G_2 := \langle (0,-\frac{1}{2(2k+1)},\frac{1}{2k+1}) \rangle$.

By the triviality of 2--spin FJRW theory and sum of singularities axiom, FJRW theory of $(\widetilde f,\widetilde G)$ is equivalent to FJRW theory of $(W,J_W)$ with $W := z_2^2z_3+z_3^{2k+1}$.

Such polynomial defines $D_{2k+2}$ type singularity.
The proof now follows from Theorem~\ref{theorem: FJR-ADE}.

\subsection{Case 8}
We compute FJRW theory of $\widetilde f = z_1^2 + z_2^{k-1}z_3 + z_3^2$ with the symmetry group $J$.
Consider the coordinate transformation
\begin{equation*}
     u = z_3 + \frac{1}{2}z_2^{k-1}.
\end{equation*}
In the coordinates $z_1,z_2,u$ the polynomial given reads
\begin{equation*}
    W = z_1^2 - \frac{1}{4}z_2^{2k-2} + u^2.
\end{equation*}
This coordinate change is quasihomogeneous. Hence FJRW theory of $(\widetilde f,J_{\widetilde f})$ is equivalent to FJRW theory of $(W,J_W)$. The latter was investigated in Case~3 above.

%%%%%%%%%%%%%%%%%%%%%%%%%%%%%%%%%%%%%%%%%%%%%%%%%%%%%%%%%%%%%%%%%%%%%%%%%%%%%%%%%%%%%%
%%%%%%%%%%%%%%%%%%%%%%%%%%%%%%%%%%%%%%%%%%%%%%%%%%%%%%%%%%%%%%%%%%%%%%%%%%%%%%%%%%%%%%
%%%%%%%%%%%%%%%%%%%%%%%%%%%%%%%%%%%%%%%%%%%%%%%%%%%%%%%%%%%%%%%%%%%%%%%%%%%%%%%%%%%%%%

\end{document}